\documentclass{amsart}
\usepackage{amsmath}
\usepackage{graphicx}
\usepackage{amsfonts}
\usepackage{amssymb}
\newtheorem{theorem}{Theorem}
\theoremstyle{plain}

\newtheorem{example}{Example}

\newtheorem{proposition}{Proposition}
\newtheorem{remark}{Remark}

\numberwithin{equation}{section}

\begin{document}
\title{On the homology theory of the closed geodesic problem}
\author{Samson Saneblidze}
\address{A. Razmadze Mathematical Institute\\
Department of Geometry and Topology\\
M. Aleksidze st., 1\\
0193 Tbilisi, Georgia} \email{sane@rmi.ge}
\date{}

\begin{abstract}
Let $\Lambda X$ be the free loop space on  a  simply connected finite $CW$-complex $X$
 and $\beta_{i}(\Lambda X;\Bbbk)$ be
 the cardinality of a minimal generating set of
$H^{i}(\Lambda X;\Bbbk)$ for $\Bbbk$ to be a commutative ring with unit. The sequence $
\beta_{i}(\Lambda X;\Bbbk) $  grows  unbounded if and only if $\tilde {H}^{\ast}(X;\Bbbk)$
requires  at least two algebra generators. This in particular answers to a long standing
problem whether  a   simply connected closed smooth manifold has infinitely many
geometrically distinct closed geodesics in any Riemannian metric.

\end{abstract}
\maketitle

\section{Introduction}

Let $Y$ be a topological space, let $\Bbbk$ be a commutative ring with unit, and
assume that the $i^{th}$-cohomology group $H^{i}(Y;\Bbbk)$ of $Y$ is finitely generated
as a $\Bbbk$-module. We refer to the cardinality of a minimal generating set of
$H^{i}(Y;\Bbbk),$ denoted by $\beta_{i}(Y;\Bbbk),$ as the \emph{generalized}
$i^{th}$\emph{-Betti number} \emph{of} $Y.$ Let $\Lambda X$ denote the free loop space,
i.e., all continuous maps from the circle $S^1$ into $X.$ In \cite{Gromoll-Meyer}
Gromoll and Meyer proved the following

\vspace{0.1in} \noindent{\bf Theorem.} \emph{Let $X$ be a simply connected closed
smooth  manifold of dimension greater than 1 and let $\Bbbk$ be a field of characteristic zero. If
 the Betti numbers  $\beta_{i}(\Lambda X;\Bbbk) $ grow unbounded, then $X$ has
infinitely many geometrically distinct closed geodesics in any Riemannian metric.}
\vspace{0.1in}

\noindent In fact,  the proof of the theorem easily shows that the statement  remains to be true for the Betti numbers $\beta_{i}(\Lambda
X;\Bbbk) $ with respect to any field $\Bbbk,$ too.
Thus,   this result has motivated  a question, the 'closed geodesic problem'\!, to find simple
criteria  which imply that  the Betti numbers $\beta_{i}(\Lambda X;\Bbbk) $ are
unbounded. Below we state such criteria in its most general form in the following
\begin{theorem}
\label{bettiF} Let $X$ be a simply connected space and $\Bbbk$ a commutative
ring with unit. If $H^{\ast}(X;\Bbbk)$ is finitely generated as a $\Bbbk$-module and
$H^{\ast}(\Lambda X;\Bbbk)$ has finite type, then the generalized Betti numbers
$\beta_{i}(\Lambda X;\Bbbk) $ grow unbounded if and only if $\tilde
{H}^{\ast}(X;\Bbbk)$ requires at least two algebra generators.
\end{theorem}
Theorem \ref{bettiF} was proved by Sullivan and Vigu\'{e}-Poirrier \cite{Sul-Vigue}
over fields of characteristic zero, and then it was conjectured for $\Bbbk$ to be a
field of positive characteristic.
  A number of papers
\cite{ziller}, \cite{Lsmith},  \cite{mcClearyGot}, \cite{mcClearyLNM}, \cite{mcCleary-ziller},
\cite{Roos},
     \cite{Halp-Vigue},   \cite{Ndombol-Thomas} deals with this conjecture
      but it remained to be open for
 $X$ to be a finite $CW$-complex and $\Bbbk$  a finite field.

Here we prove Theorem \ref{bettiF}. More precisely, it is a consequence of the following more general
algebraic fact:
 Let $A=\{ A^{i} \},i\in {\mathbb Z},$  be a torsion free graded abelian
group such that $A$ is an associative Hirsch algebra and the bar construction $BA$ is a Hirsch algebra \cite{saneFiltered}; this in particular
means that
 $A$ is an (associative) graded differential  algebra (dga) endowed with higher order  operations $E=\{E_{p,q}\}$ and $E'=\{E'_{p',q'}\}$
 such that $E$ induces  an associative  product $\mu_{E}$ on the bar construction $BA$ converting it into a dg Hopf algebra, and, similarly, $E'$
induces a product $\mu_{E'}$ (not necessarily associative) on the double bar construction $B^2A.$
(A major component of $E'$ is a binary product
$E'_{1,1}$ on $A$    measuring the non-commutativity of the operation $E_{1,1};$ note that
in the topological setting of $A$ such a binary product is just provided by Steenrod's  cochain $\smile_2$-product.)
Below the algebra $A$ is referred to as a \emph{special} Hirsch algebra.
Let  $A_{\Bbbk}=A\otimes_{\mathbb{Z}}\Bbbk.$

We have the following theorem whose
proof appears in Section \ref{theorem}:
\begin{theorem}
\label{varsigma} Assume that $H^{\ast}(A_{\Bbbk})$ is finitely generated as a $\Bbbk$-module
and that the Hochschild homology $HH_*(A_{\Bbbk})$ has finite type. Let $\varsigma_{i}(A_{\Bbbk})$
denote the cardinality of a minimal generating set of $HH_i(A_{\Bbbk}).$ Then the integers $
\varsigma_{i}(A_{\Bbbk})$ grow unbounded if and only if $\tilde{H}^*(A_{\Bbbk})$ requires at least two algebra
generators.
\end{theorem}

Let $C^{\ast}(X;\Bbbk)=C^{\ast}(\operatorname{Sing}^{1}X;\Bbbk)/C^{>0}
(\operatorname{Sing}\,x\,;\Bbbk)$ in which ${\operatorname{Sing}}^{1}
X\subset{\operatorname{Sing}}X$ is the Eilenberg 1-subcomplex generated by the singular
simplices that send the 1-skeleton of the standard $n$-simplex $\Delta^{n}$ to the base
point $x$ of $X.$
 Theorem \ref{bettiF} is deduced from Theorem \ref{varsigma} by setting
$A_{\Bbbk}=C^{\ast}(X;\Bbbk)$: Indeed, $C^{\ast}(X;\Bbbk)$ is a special Hirsch algebra (\cite{baues}, \cite{KScubi})
  and there is  the isomorphism $HH_*(C^*(X;\Bbbk))\approx H^*(\Lambda X;\Bbbk)$ (\cite{Jones}, \cite{saneFREE}).
 When $\tilde {H}^*(A_{\Bbbk})$ requires  at least two algebra
generators, we construct two infinite sequences in the Hochschild homology $HH_*(A_{\Bbbk})$
and take all possible products of their components to detect a submodule of $HH_*(A_{\Bbbk})$
at least as large as the polynomial algebra $\Bbbk[x,y].$ The construction  of these
sequences is in fact  based  on the notion
of a formal $\infty$-implication sequence  \cite{saneBetti}  that  generalizes W. Browder's
notion of $\infty$-implications \cite{browder}. As in \cite{saneBetti},  we use a filtered Hirsch model of $A_{\Bbbk}$ this time
 to construct a small model for the Hochschild chain complex of $A_{\Bbbk}$ and then to reduce the chain product \cite{saneFREE} on this model
  inducing
  the aforementioned product on  $HH_*(A_{\Bbbk}).$
While the constructions of the sequences  in our both papers   are similar,  here is an essential exception  that  we have to detect  a desired sequence   in the kernel of the canonical
homomorphism
$HH_*(A_{\Bbbk})\rightarrow H^*(BA_{\Bbbk})$ (corresponding to $i^*:H^*(\Lambda X;\Bbbk)\rightarrow H^*(\Omega X;\Bbbk)$ for $i:
\Omega X \hookrightarrow \Lambda X $);
also  some technical details are simplified.

Though the author  has been  considered several special cases
of Theorem \ref{bettiF} during the last two decades  but it was just recently
the integer coefficients come into play: In particular,  the filtered Hirsch model over the integers controls  the subtleties when dealing with the Bockstein homomorphism
in question.

I am grateful to Edward Fadell for discussing about the subject and in particular for
pointing out the paper \cite{Sul-Vigue} when the author was visiting the Heidelberg
University at the beginning of 90's.

\section{Some preliminaries and conventions}

We adopt some basic  notations and terminology of \cite{saneFiltered}. We fix a ground commutative
ring $\Bbbk$ with unit and
let $\mu>0$ denote the smallest integer such that
$\mu\kappa=0$ for all $\kappa\in\Bbbk.$ When such a positive integer does not exist, we assume $\mu=0.$
Let $A^*=\tilde{A}\oplus \Bbbk$ be a supplemented dga. In general $A^*$ may be graded over the integers  ${\mathbb Z}.$
Assuming  $A$ to be associative,
the (reduced)
bar construction $BA$  is the tensor coalgebra $T(\bar A),\ \bar A=
s^{-1}\tilde A,$ with differential $d_{BA}=d_{1} +d_{2} $ given for $[\bar a_{1}|\dotsb|\bar
a_{n}] \in T^{n}(\bar A)$ by
\begin{equation*}
d_{1}[\bar a_{1}|\dotsb|\bar a_{n}]=-\sum_{1\leq i\leq n} (-1)^{\epsilon^a_{i-1}}[\bar
a_{1}|\dotsb|\overline{d_{A}(a_{i})}|\dotsb|\bar a_{n}]
\end{equation*}
and
\begin{equation*}
d_{2} [\bar a_{1}|\dotsb|\bar a_{n}]=- \sum_{1\leq i<n} (-1)^{\epsilon^a_{i}}[\bar
a_{1}|\dotsb|\overline{a_{i}a_{i+1}}|\dotsb|\bar a_{n}],
\end{equation*}
where $\epsilon^x_{i}=|x_{1}|+\cdots +|{x_{i}}|+i.$
An associative  dga $A$ equipped with multilinear maps $E=\{E_{p,q}\}_{p+q>0},$
\[
E_{p,q}:A^{\otimes p}\otimes A^{\otimes q}\to A, \ \ \ \  \ \ \  E_{1,0}=Id=
E_{0,1},\ E_{p>1,0}= 0=E_{0,q>1},
\]
of degree $1-p-q$ is a \emph{Hirsch} algebra if
$E$ lifts to a dg coalgebra map $\mu_{E}:BA\otimes BA \rightarrow BA.$ A basic ingredient of $E$ is the binary operation $\smile_1:=E_{1,1}$
measuring the non-commutativity of the $\cdot$ product on $A$ by the formula
\[
d(a\smile_1 b)- da\smile_1 b   +(-1)^{|a|} a\smile_1 db= (-1)^{|a|}ab -(-1)^{|a|(|b|+1)}ba.
\]
Given a Hirsch algebra $(A, \{E_{p,q}\})$ with  $H=H(A),$
there is its  filtered Hirsch model
\begin{equation}\label{fhmodel}
f:(RH,d_{h})\rightarrow(A,d_{A})
\end{equation}
in which $\rho: (RH,d)\rightarrow H$ is a multiplicative resolution of the commutative graded algebra (cga)  $H:$
As a module  each row of $R^*H^*$ for $m\in {\mathbb Z}$
           \[\cdots \overset{d}{\rightarrow}R^{-2}H^m\overset{d}{\rightarrow} R^{-1}H^m \overset{d}{\rightarrow}  R^0H^m\overset{\rho}{\rightarrow} H^m\]
  represents  a  free  resolution of a $\Bbbk$-module $H^m.$
As an algebra
 $R^*H^*=T(V^{*,*})$ is a (bi)graded tensor algebra with
 \begin{equation*}
V^{*,*}= \mathcal{E}^{*,*}\oplus U^{*,*}= \mathcal{E}^{*,*} \oplus \mathcal{T}^{*,*}\oplus {\mathcal M}^{*,*};
\end{equation*}
 the module
  $V^{0,*} ={\mathcal M}^{0,*}$ corresponds to a choice of  multiplicative generators of $H,$ while  ${\mathcal M}^{-1,*}$
  to   relations among them  which is not a consequence of that of the commutativity of the algebra $H,$
  and then ${\mathcal M}^{-r,*}$  for  $r\geq 2$ is defined by the syzygies.
  The module
 ${\mathcal E}^{<0,*}$ just corresponds to the commutativity relation in $H$ and is formed by the products under all operations
 $E_{p,q}$ on $RH.$
 In particular, \[V^{-1,*}=\mathcal{E}^{-1,*}\oplus {\mathcal M}^{-1,*}\]
 where $\mathcal{E}^{-1,*}$ is formed by the products $a\smile_1 b$ for $a,b\in R^0H^*,$ while ${\mathcal M}^{-1,*}\neq 0$
  for $H$ to be a non-free cga (e.g.
    $\dim H^* <\infty$
  and $H^{ev}\neq 0;$ see (\ref{relation})--(\ref{power}) below).
 The module ${\mathcal T}$ is determined  by the $\cup_2$-product that measures the non-commutativity of the $\smile_1$-product, so that its first non-trivial component ${\mathcal T}^{-2,*}$ contains the products $a\cup_2 b$ for $a,b\in V^{0,*}.$ More precisely,
   $(RH,d)$
  is also endowed with Steenrod's type binary operation, denoted by $\smile_2,$
 so that the (minimal) Hirsch resolution $(RH,d)$ can be viewed as  a special Hirsch algebra
with $E=\{E_{p,q}\}$ and $E'$ consisting of a  single operation  $\smile_2:=E'_{1,1}.$ In particular,   the relationship between $a\cup_2b$ and $a\smile_2b$ for
$a,b\in {\mathcal V}$ with $da,db\in V,\,$ where  ${\mathcal V}$ is a basis of $V,$
is given by
\[
a\smile_2b=\left\{\begin{array}{llll}
a\cup_2b, & a\neq b,    \\
2 a \cup_2a,& a=b ;
\end{array}
\right.
\]
thus  $d(a\cup_2 a)=a\smile_1a $ for $a$ to be of even degree. ($a\cup_2a=0$ for an odd dimensional $a\in RH.$)
Regarding the differential $d_h$ on $RH,$ we have
\[
d_{h}=d+h,\ \ \ h=h^{2}+\cdots+h^{r}+\cdots,\ \ \ h^{r}:R^{p}H^{q}\rightarrow R^{p+r}H^{q-r+1}.
\]
 Given $r\geq 2,$ the map $h^r|_{R^{-r}H}:R^{-r}H\rightarrow R^0H$ is referred to as the \emph{transgressive} component of
$h$ and is denoted by $h^{tr}.$ The perturbation $h$ is extended as a derivation on ${\mathcal E}$ so that $h^{tr}({\mathcal E})=0.$

 Furthermore, if $A$ is also a special Hirsch algebra in (\ref{fhmodel}), we can simply choose $h$ and $f$ such that
 \begin{equation}\label{cup2}
h^{tr}(a\cup_2b)=0  \ \  \text{for}\ \  a\neq b \ \  \text{in}\ \  {\mathcal V}.
\end{equation}
Just to achieve   this equality in $(RH,d_h),$  we have in fact evoked  the product  $\mu_{E'}$ on $B^2A$ (cf. \cite[Proposition 4]{saneFiltered}).

A Hirsch resolution $(RH,d)$ is \emph{minimal} if
\[
d(u)\in \mathcal{E}+\mathcal{D}+\kappa_u\!\cdot\!V\ \   \text{for}\ \ u\in U
\]
where ${\mathcal{D}^{\ast,\ast}}\subset R^{\ast}H^{\ast}$ denotes the submodule of
decomposables $RH^{+}\!\cdot RH^{+}$ and $\kappa_u\in\Bbbk$ is non-invertible; for
example, $\kappa_u\in {\mathbb Z}\setminus \{-1,1\}$ when $\Bbbk=\mathbb{Z}$ and
$\kappa_u=0$ for all $u$ when $\Bbbk$ is a field.

 In the sequel $A$ denotes
 a  torsion free special Hirsch ${\mathbb Z}$-algebra, while $A_{\Bbbk}=A \otimes_{\mathbb{Z}}\Bbbk$ and $H_{\Bbbk}=H(A_{\Bbbk}).$ Assume $(RH,d)$ is minimal and
let $RH_{\Bbbk }=RH\otimes_{\mathbb{Z}}\Bbbk;$ in particular, $RH_{\Bbbk}=T(V_{\Bbbk})$
for $V_{\Bbbk}=V\otimes_{\mathbb{Z}}\Bbbk.$ When $\Bbbk$ is a field of characteristic
zero, $H_{\Bbbk}=H\otimes \Bbbk$
and  $\rho_{\Bbbk}=\rho\otimes1:RH_{\Bbbk}\rightarrow H_{\Bbbk}$ is a Hirsch
resolution of $H_{\Bbbk},$ which is \emph{not }minimal when $\operatorname{Tor}H\neq0$.
Assuming $A$ is   ${\mathbb Z}$-algebra in (\ref{fhmodel})
   we obtain  a  Hirsch model of  $(A_{\Bbbk},d_{A_{\Bbbk}})$
 as
\[
f_{\Bbbk}=f\otimes1:(RH_{\Bbbk},d_{h}\otimes1)\rightarrow (A_{\Bbbk},d_{A_{\Bbbk}}).
\]

\subsection{Small Hirsch resolution}\label{small}
In practice it is convenient  to reduce the Hirsch resolution $RH$ at the cost of the module ${\mathcal E}.$ Here we
define such a small resolution $R_{\tau}H$ (compare with $R_{\varsigma}H$ in \cite{saneFiltered}).
Namely,
 let
  \[R_{\tau}H=RH/ J_{\tau} \]
 where $J_{\tau}\subset RH$ is a Hirsch ideal generated by
   \begin{multline*}
   \hspace{-0.1in}\{ E_{p,q}(a_1,...,a_p;a_{p+1},...,a_{p+q}),\,
   dE_{1,2}(a_1;a_2, a_3),\,
   dE_{2,1}(a_1,a_2; a_3),\, a\cup_2 b,\,d(a\cup_2b) \\
   |\,  (p,q)\neq(1,1),\,
     a,b\in {\mathcal V}, \, a\neq b\}
   \end{multline*}
   where   $a_i\in RH$ unless $i=p+q$ for  $p\geq 2$ and $q=1$ in which case $a_{p+1}\in {\mathcal V}.$
   Since $d: J_{\tau}\rightarrow  J_{\tau},$ we get a Hirsch algebra surjection  $g_{\tau}:(RH,d){\rightarrow} (R_{\tau}H,d)$  so  that
    a resolution map $\rho: RH \rightarrow H$   factors  as
   \[\rho: (RH,d)\overset{g_{\tau}}{\longrightarrow} (R_{\tau}H,d) \overset{\rho_{\tau}}{\longrightarrow} H.\]
    By definition we have $h:{\mathcal E}\rightarrow {\mathcal E};$ this fact together with (\ref{cup2})
     implies
 $h: J_{\tau}\rightarrow  J_{\tau}.$  Thus
   $g_{\tau}$  extends to  a quasi-isomorphism  of Hirsch algebras
 \begin{equation*}
 g_{\tau}:(RH,d_h)\rightarrow (R_{\tau}H,d_h).
 \end{equation*}

We have that the  Hirsch algebra structure of   $(R_{\tau}H,d_h)$ is given by the $\smile_1$-product satisfying the following two  formulas.
The (left) Hirsch formula: For $a,b,c\in R_{\tau}H,$
\begin{equation} \label{hirsch1}
c\smile _{1}ab=(c\smile _{1}a)b+(-1)^{(|c|+1)|a|}a(c\smile _{1}b)
\end{equation}
\noindent and the (right) generalized Hirsch formula: For
$a,b\in R_{\tau}H$  and $c\in V_{\tau}$
with
$d_h(c)=\sum c_1\cdots c_q,\,c_i\in V_{\tau},$
\begin{equation}
\label{hirsch2} ab\smile_1 c=\left\{\!\!\begin{array}{llll} a(b\smile_1 c)+
(-1)^{\varepsilon_1}
(a\smile_1 c)b, &  q=1,\vspace{5mm}\\
 a(b\smile_1 c)+(-1)^{\varepsilon_1}(a\smile_1 c)b\vspace{1mm}\\
\hspace{0.62in} +\underset{1\leq i<j\leq q}{\sum}(-1)^{\varepsilon_2}\,c_1\cdots c_{i-1}(a\smile_1 c_i)c_{i+1}\vspace{1mm}\\
\hspace{1.5in}\cdots c_{j-1}(b\smile_1 c_{j})c_{j+1}\cdots c_q, & q\geq 2, \\
\end{array}
\right.
\end{equation}
\hspace{1.0in}$\varepsilon_1=|b|(|c|+1),\, \varepsilon_2=(|a|+1)(\epsilon^c_{i-1}+i+1)+(|b|+1)(\epsilon^c_{j-1}+j+1).$
\begin{remark}
1. Formula $(\ref{hirsch2})$ can be thought of as a generalization of Adams' formula for
the $\smile_1$-product in the cobar construction \cite[p.\,36]{adams-hopf1} from $q=2$ to any $q\geq 2.$

2. We just pass from $RH$ to $R_{\tau}H$ to have formulas $(\ref{hirsch1})$--$(\ref{hirsch2})$ therein; more precisely, we use them together with the commutativity
of $a\smile_1b$ for $a,b\in V_{\tau}$  to build  the sequence given by $(\ref{even})$ below.
\end{remark}

\subsection{Cohomology operation $\mathcal{P}_1$}

Let $\mu\geq 2.$ Given an element $a\in A_{\Bbbk}$ and the integer $n\geq 2,$ take (the right most) $n^{th}$-power
 of $\bar{a}\in \bar{A}_{\Bbbk}\subset BA_{\Bbbk}$ under the $\mu_E$ product on $BA_{\Bbbk}$
  and then consider its component in $\bar{A}_{\Bbbk}.$ Denote this component  by
  $s^{-1}(a^{\uplus n}) $ for $a^{\uplus n}\in A_{\Bbbk}.$ The element  $a^{\uplus n}$  has the form
\[
a^{\uplus n}=a^{\smile_1n }+ Q_n(a),
\]
where   $Q_n(a)$  is expressed  in terms of $E_{1,k}$ for $1<k <n$ so that
 $Q_2(a)=0,$ i.e., $a^{\uplus 2}=a^{\smile_1 2}.$
    In particular, if $E_{1,k}=0$ for $k\geq 2$ (e.g.   $A_{\Bbbk}$ is a homotopy Gerstenhaber algebra (HGA)), then $a^{\uplus n}=a^{\smile_1 n }.$

Let $p\geq 2$ be  the smallest prime that divides $\mu.$
Define the cohomology operation
$\mathcal{P}_1$ on $H_{\Bbbk}$  as follows.

For $p$ odd:
\[\hspace{-0.23in} \mathcal{P}_1: H_{\Bbbk}^{2m+1}\rightarrow H^{2mp+1}_{\Bbbk},\ \ \ [a] \rightarrow
\left[\frac{\mu}{p}a^{\uplus p}\right],  \]

 for $p=2:$
 \[\hspace{1.1in}\mathcal{P}_1: H_{\Bbbk}^{m}\rightarrow H^{2m-1}_{\Bbbk},\hspace{0.4in} [a]\rightarrow
\left[\frac{\mu}{2}a\smile_1\! a\right],\ \ \ \, da=0,\    a\in \!A_{\Bbbk}. \]
Since $a\smile_1a$ is a cocycle in $A_{\Bbbk}$ for an even dimensional cocycle $a$ independently on the parity of $p,$
we also set $\mathcal{P}_1[a]=[a\smile_1a]$ for $p$ odd and $\mu\geq 0;$
 obviously,  $\mathcal{P}_1[a]=0.$
Let
 $\mathcal{P}^{(m)}_1$  denote  the $m$-fold composition
        $\mathcal{P}_1\circ \cdots\circ \mathcal{P}_1.$
        Given  $x\in H_{\Bbbk},$ let
    $\nu\geq 0$ be the smallest integer such that
    $\mathcal{P}^{(\nu+1)}_1( x)=0.$ The integer
    $\nu$ is referred to as
\emph{$\smile_1$-height} of $x.$

\section{small model for the Hochschild chain complex}

Given an associative   dga $C,$
    its  (normalized) \emph{Hochschild chain complex}  $\Lambda C$
is  $C\otimes  {B}C$ with  differential $d_{{\Lambda C}}$ defined by $d_{{\Lambda C}}=
d_C\otimes 1 + 1 \otimes d_{ BC}+ \theta ^1+\theta ^2 ,$ where
\begin{equation*}
\begin{array}{lll}
\theta ^1(u\otimes [\bar a_1|\dotsb |\bar a_n]) = -(-1)^{|u|}ua_1\otimes [\bar
a_2|\dotsb |\bar a_n],\newline $\vspace{1mm}$ \\

\theta ^2 (u\otimes [\bar a_1|\dotsb |\bar a_n]) =(-1)^{(| a_n|+1)
(|u|+\epsilon^a_{n-1})} a_n u\otimes [\bar a_1|\dotsb |\bar a_{n-1}].
\end{array}
 \end{equation*}
The homology of  $\Lambda C$ is called the Hochschild homology of  $C$ and is
denoted by $HH_*(C).$

Let $C=T(V_{\Bbbk})$ be a tensor algebra with $V^*_{\Bbbk}$ a free $\Bbbk$-module.
 Denote \[\bar{V}_{\Bbbk}=s^{-1}(V_{\Bbbk}^{>0})\oplus\Bbbk. \]
 Then $\Lambda C$ can be replaced by the small complex
$(C\otimes\bar{V}_{\Bbbk},d_{\omega})$
 where the differential $d_{\omega}$ is defined as follows (cf. \cite{Vigue}, \cite{J-M.hoh}):
\begin{multline*}   d_{\omega}(u\otimes \bar a)=
d_{_C}(u)\otimes \bar a-
(-1)^{|u|}(1\otimes s^{-1})\chi(u\otimes d_{_C}(a))\\ -(-1)^{|u|+|a|}(ua-(-1)^{|a||u|}au)\otimes 1,
\end{multline*}
 in which \[\chi:C\otimes C\rightarrow C\otimes V_{\Bbbk}  \]
  is a map given for $u\otimes a\in C\otimes C$ with  $a=a_1\cdots a_n,\,a_i\in V_{\Bbbk},$     by
\[
\chi(u\otimes a)=\left\{
\begin{array}{llll}
0,&&  a= 1,\\

u\otimes a, && n=1,\vspace{1mm}\\

\underset{1\leq i\leq n}{\sum}(-1)^{\varepsilon} a_{i+1}\cdots a_{n}\, u\, a _1\cdots a_{i-1}\otimes a_i,& &    n\geq 2,
\end{array}
\right.
\]
$\hspace{2.2in} \varepsilon=(|a_{i+1}|+\cdots +|a_n|)(|u|+|a_1|+\cdots +|a_i|).$

There is a chain map
\begin{equation}\label{phi}
\phi:(\Lambda C,d_{\Lambda C})\rightarrow (C\otimes\bar{V}_{\Bbbk},d_{\omega})
\end{equation}
defined for    $u\otimes x\in \Lambda C$   by
\[\phi(u\otimes x)=\left\{
\begin{array}{llll}
u\otimes 1,  && x= [\  ],\\

(-1)^{|u|}(1\otimes s^{-1})\chi(u\otimes  a), && x =[\bar a],\\

0, &&  x= [\bar a_1|\cdots |\bar a_n],& n\geq 2,
\end{array}
\right.
\]
and $\phi$ is a homology isomorphism.

  For $C=RH_{\Bbbk},$ define
the differential $\bar{d}_{h}$ on $\bar{V}_{\Bbbk}$ by the restriction of $d_{h}$ to
$V_{\Bbbk}$ to obtain the cochain complex $(\bar{V}_{\Bbbk },\bar{d}_{h}).$ There are the sequences of maps
   \[(\bar V_{\tau})_{\Bbbk} \overset{\psi}{\longleftarrow}   B (R_{\tau}H_{\Bbbk}) \overset{B (g_{\tau})_{\Bbbk}}{\longleftarrow}
   B (RH_{\Bbbk})\overset{\!\!B f_{\Bbbk}}{\longrightarrow}    B A _{\Bbbk} \]
and
\[ R_{\tau}H\otimes (\bar V_{\tau})_{\Bbbk} \overset{\phi}{\longleftarrow}   \Lambda (R_{\tau}H_{\Bbbk}) \overset{ \Lambda  (g_{\tau})_{\Bbbk}}{\longleftarrow}
    \Lambda  (RH_{\Bbbk})\overset{ \!\!\Lambda  f_{\Bbbk}}{\longrightarrow}     \Lambda  A _{\Bbbk} \]
subjected to the following proposition
\begin{proposition}
\label{barV}
There are isomorphisms of $\Bbbk$-modules
\begin{multline*}
H^{*}((\bar V_{\tau})_{\Bbbk},\bar{d}_{h})\overset{\psi^*}{\underset{\approx}{\longleftarrow}} H^{*}(B(R_{\tau}H_{\Bbbk}),d_{_{B(R_{\tau}H_{\Bbbk })}})
\overset{B(g_{\tau})_{\Bbbk}^*}{\underset{\approx}{\longleftarrow}}
H^{*}(B(RH_{\Bbbk}),d_{_{B(RH_{\Bbbk })}})
\\
\overset{\!\!B f^*_{\Bbbk}}
{\underset{\approx}{\longrightarrow}}
H^{*}(BA_{\Bbbk},d_{_{BA_{\Bbbk}}})
\end{multline*}
and
\begin{multline*}
H^{*}(R_{\tau}H\otimes (\bar V_{\tau})_{\Bbbk},{d}_{\omega})
\overset{\phi^*}{\underset{\approx}{\longleftarrow}}
 H^{*}(\Lambda( R_{\tau}H_{\Bbbk}),d_{_{\Lambda (R_{\tau}H_{\Bbbk})}})
\overset{\Lambda (g_{\tau})_{\Bbbk}^*}{\underset{\approx}{\longleftarrow}}
H^{*}(\Lambda( RH_{\Bbbk}),d_{_{\Lambda (RH_{\Bbbk})}})\\
\overset{\!\!\Lambda f_{\Bbbk}^*}{\underset{\approx}{\longrightarrow}}
 H^{*}(\Lambda A,d_{_{\Lambda A}}).
\end{multline*}
\end{proposition}
Note that the isomorphism $\psi^*$ above  is a consequence of a general fact about tensor algebras \cite{F-H-T}.
Recall  also the following isomorphisms
  $H^{\ast}(BC^{\ast }(X;\Bbbk),d_{_{BC}})\approx H^{\ast}(\Omega X;\Bbbk)$ (\cite{baues}) and
 $H^{*}(\Lambda C^{*}(X;\Bbbk),d_{_{\Lambda C}}) \approx H^{*}(\Lambda X;\Bbbk)$ (\cite{Jones}, \cite{saneFREE}) to deduce the following  proposition for
  $A_{\Bbbk}=C^{\ast}(X;\Bbbk).$
\begin{proposition}
\label{barVX} There are isomorphisms of $\Bbbk$-modules
\[
H^{*}((\bar V_{\tau})_{\Bbbk},\bar{d}_{h})\approx H^{*}(BC^{*}(X;\Bbbk),d_{_{BC}}) \approx
H^{*}(\Omega X;\Bbbk)\] and
\[
H^{*}(R_{\tau}H\otimes (\bar V_{\tau})_{\Bbbk},{d_{\omega}})\approx H^{*}(\Lambda
C^{*}(X;\Bbbk),d_{_{\Lambda C}}) \approx H^{*}(\Lambda X;\Bbbk).
\]
\end{proposition}
In the sequel by abusing the notations we will  denote  $R_{\tau}H$  again  by $RH.$

\subsection{Product on the small model of the  Hochschild chain complex}
Let $RH=T(V)$ be the (small) Hirsch resolution with only  $\smile_1$-product.
 First, define  a product on   $\bar V$ for $\bar a, \bar b\in \bar V$
 by
\[\bar a\bar b=\overline {a\smile_1b} \ \ \  \text{with}\ \ \     \bar a1=1\bar a=\bar a.
\]
 Next define a product on
$RH\otimes \bar V$  for  $u\otimes \bar a, v\otimes \bar b\in RH\otimes \bar V$ with
$d_ha=\sum a_1a_2\!\mod V,$  $d_hb=\sum b_1b_2\!\mod V$ and $a_2,b_1\in V,$
$a_1,b_2\in RH$ by
\begin{multline*}
 ( u\otimes \bar a)(v\otimes \bar b)
 =
 (-1)^{\epsilon_1} uv\otimes   \overline{a\smallsmile_1 b}
 +u( a\smallsmile_1 v)\otimes \bar b
 + (-1)^{\epsilon_2}
  (u\smallsmile_1b)v\otimes \,\bar a
\\
\hspace{3.0in}+(-1)^{\epsilon_3}
  (u\smallsmile_1 b)(a\smallsmile_1 \!v)\otimes 1\\
  \hspace{1.0in}-\!\!\sum  (-1)^{\epsilon_4}u(a_1\smallsmile_1\! v)\otimes   \overline{a_2\smallsmile_1 b}
   \, +(-1)^{\epsilon_5}(u\smallsmile_1\! b)(a_1\smallsmile_1\! v) \otimes \bar a_2\\
   \hspace{1.1in}
       +\sum (-1)^{\epsilon_6}\!   (u\smallsmile_1 \!\! b_2)v\,\, \otimes\     \overline{a\smallsmile_1\! b_1}
    \,\,\,\, +\,\, (-1)^{\epsilon_7}\!  (u\smallsmile_1\!\! b_2)(a\smallsmile_1\!\! v)\, \otimes\, \bar b_1,
\end{multline*}
\[
\begin{array}{lll}
\hspace{0.5in}\epsilon_1=(|a|+1)|v|,& \epsilon_4= |a_1||b|+(|a_2|+1)|v|  ,\\

\hspace{0.5in}\epsilon_2=|a|(|v|+|b|+1)+|v||b|,& \epsilon_5=|a_2|(|v|+1)+(|a|+|v|)|b|,
\\
\hspace{0.5in}\epsilon_3=(|a|+|v|)(|b|+1),& \epsilon_6= (|a|+|b_2|)(|v|+1)+(|a|+|b_1|)|b_2|  , \\

  &\epsilon_7=(|a|+|v|)(|b_2|+1)+(|b_1|+1)|b_2|,
  \end{array}
  \]
and
\begin{equation*}
 \begin{array}{llll}
( u\otimes 1)(v\otimes 1)&=& uv\otimes 1\\
( u\otimes 1)(v\otimes \bar b) &= & uv\otimes \bar
b+(-1)^{(|v|+1)(|b|+1)}(u\smallsmile_1b) v
\otimes 1,\\
( u\otimes \bar a)(v\otimes 1)& = & (-1)^{(|a|+1)|v|}uv\otimes \bar a+u(a\smallsmile_1
v)\otimes 1.
\end{array}
\end{equation*}
  Formulas (\ref{hirsch1})--(\ref{hirsch2}) guarantee that such defined products
   satisfy the  Leibniz rule, and  we obtain a short sequence of dg algebras
  \begin{equation}
\label{short}
  \bar V_{\Bbbk} \overset{\pi}{\leftarrow} RH\otimes \bar V_{\Bbbk}\overset{\iota}{\leftarrow} RH_{\Bbbk}
   \end{equation}
with $\iota$ and $\pi$ the standard inclusion and projection respectively.

\begin{remark}
1. The dga $(RH\otimes \bar V_{\Bbbk},d_{\omega})$ can be thought of as a non-commutative version of the cdga
$\left({\mathcal A}^*(X)\otimes H^*(\Omega X;{\mathbb Q}),\bar d\, \right)$
modeling $\Lambda X$ \cite{Sul-Vigue}.

2. Taking into account the product on the Hochschild chain complex $\Lambda A$ of $A=C^*(X;\Bbbk)$  defined in \cite{saneFREE} one can show that the map
$\phi$ given by $(\ref{phi})$ is multiplicative up to homotopy;
thus the   sequence given by $(\ref{short})$  provides a small multiplicative model of the free loop fibration.
\end{remark}

\section{Canonical  sequences  in $RH\otimes \bar V$}
Motivated by the notion
 of a
formal  $\infty$-implication sequence \cite{saneBetti} here we construct certain
sequences in the dga $(RH\! \otimes \bar V,d_{\omega})$ used in the proof of Theorem\,\,2.
 First,
 we consider a more general situation.

\subsection{The sequence $\mathbf{x}_\mu$ in $(C,d_{_{C}})$}\label{twosequences}
Let $(C^*,d_{_{C}})$ be a cochain complex of torsion free abelian groups and
let
\[t_{\Bbbk}:C\rightarrow C_{\Bbbk}  \,\,(=C\otimes_{\mathbb Z}\Bbbk) \]
 be the standard map.
Let $x\in C$ be a $\!\!\!\!\mod\! \mu$ cocycle, i.e., $d_{_{C_{\Bbbk}}}(t_{\Bbbk}x)=0.$ Consider for $x$   the following two conditions:
\begin{equation}\label{p1}
  [x]\neq 0 \  \text{in}\  H(C)\  \text{for}\   d_{_{C}}x=0 ;
\end{equation}
\begin{multline}\label{p2}
\text{If}\  [t_{\Bbbk}x]= 0\in H(C_{\Bbbk}), \  \text{i.e., there is a relation}\
 d_{_{C}}a=x+\lambda a', \   a,a'\in C,\,\text{then}\\
  d_{_{C_{\Bbbk}}}(t_{\Bbbk}a')=0\ \
\text{for}\, \,  \lambda\   \text{to be the greatest integer divisible by}\  \mu.
\end{multline}
Obviously, for $x$ with  $d_{_{C}}x=0$      condition (\ref{p2}) follows from (\ref{p1}).    In any case for $a'$   from
(\ref{p2}) we   have that $[t_{\Bbbk}a']\neq 0$ in $H(C_{\Bbbk}).$

Let  $\mathbf{x}=\{x(n)\}_{n\geq 0}$ be a sequence in $C^*$
 with $|x(k)|\neq |x(\ell)|$ for $k\neq \ell$
and let  $x(n)$ satisfy (\ref{p1})--(\ref{p2}) for all $n$ in which case
we also say that $\mathbf{x}$ satisfies (\ref{p1})--(\ref{p2}).
 Define the \emph{associated} sequence $\mathbf{x}_\mu=\{x_{\mu}(n)\}_{n\geq 0}$
in $C$  as follows: Given $n\geq 0,$ let
 \begin{equation*}
x_{\mu}(n)=\left\{\begin{array}{llll}
a'_n, & [t_{\Bbbk}x(n)]= 0, \vspace{1mm}\\
x(n), & [t_{\Bbbk}x(n)]\neq 0,
\end{array}
\right.
\end{equation*}
 where
$a'_n$ is resolved from
  (\ref{p2}) for $x=x(n).$
   Obviously,
 $[t_{\Bbbk}x_{\mu}(n)]\neq 0$  in $H(C_{\Bbbk})$ for all $n.$

A pair of sequences  $(\mathbf{x},\mathbf{y}')=\left(\{x(i)\}_{i\geq 0}\,,\{y'(i)\}_{i\geq
0}\right)$
satisfying (\ref{p1})--(\ref{p2}) is said to be \emph{admissible}, if $\alpha_1x(i)+\alpha_2y'(j)$ also satisfies (\ref{p1})--(\ref{p2})
whenever $ |x(i)|=|y'(j)|,$  $\alpha_1,\alpha_2\in{\mathbb Z}. $
Then  obtain the sequence $\mathbf{y}_{\mu}=\{y_{\mu}(j)\}_{j\geq0}$  from an admissible pair
 $(\mathbf{x},\mathbf{y}')$
 as follows.  Given $j\geq 0,$
  set
  \begin{equation}\label{xy'}
 y_{\mu}(j)=\left\{\begin{array}{llll}
  a'_j, &  [t_{\Bbbk}\left(\alpha_1x(i)+\alpha_2y'(j)\right)]= 0, \vspace{1mm}\\
 y'_{\mu}(j), & \text{otherwise},
    \end{array}
\right.
\end{equation}
where
$a'_j$ is resolved from (\ref{p2}) for $x=\alpha_1x(i)+\alpha_2y'(j);$
in particular,  the pair
$\left([t_{\Bbbk}x_{\mu}(i)]\,,[t_{\Bbbk}y_{\mu}(j)]\right)$ is
 linearly independent in $H(C_{\Bbbk}).$

\subsection{Sequences in $(RH\otimes \bar V,d_{\omega})$}\label{sequencesinhoh}

  Let ${\mathcal{D}}_{\Bbbk }\subset RH$ be
a subset defined by
\[
{\mathcal D}_{\Bbbk}=
\{u+\mu v\,|\,u\in\mathcal{D},v\in V\}.
\]
An element $x\in V$ with $d_{h}x\in {\mathcal{D}}+ \lambda V,\,\lambda\neq1,$ is
$\lambda$-\emph{homologous to zero} if there are $u,v\in V$ and $c\in {\mathcal D}$
such that
\[
d_{h}u=x+c+\lambda v.
\]
$x$ is \emph{weakly homologous} to zero if $v=0$ above. We have the following statement (cf. \cite{saneBetti}):
\begin{proposition}\label{nonweak}
 Let $v\in V$ and $d_{h}v\in\mathcal{D}.$ If $d_{h}v$ has a
summand component $v_1v_2\in\mathcal{D}$ such that $v_1,v_2\in{V},$ $d_{h}v_1,d_{h}v_2\in\mathcal{D},$ both ${v_1}$ and  ${v_2}$  are not weakly homologous to
zero, then  $v$ is also not weakly homologous to zero.
\end{proposition}
Note also that   under the hypotheses of the
proposition if $[\bar{v}_1],[\bar{v}_2]\neq0,$ then $[\bar {v}]\neq0$ in
$H^{\ast}(\bar{V},\bar{d}_h);$ for example, for $RH=\Omega BH_{\Bbbk}$ (the cobar-bar construction of $H_{\Bbbk}$) with $V=BH_{\Bbbk}$ and $\Bbbk$  a field,  the proposition  reflects
the obvious fact that an element $x\!\in\! H^{\ast}(BH_{\Bbbk})$ is non-zero whenever
 some $x^{\prime}\otimes x^{\prime\prime}\neq 0$ in $\Delta x=\sum
x^{\prime}\otimes x^{\prime\prime}$ for the coproduct $\Delta:BH_{\Bbbk}\rightarrow BH_{\Bbbk}\otimes BH_{\Bbbk}.$

Let \[\chi_{_1}: RH\rightarrow RH\otimes \bar V\]
 be a  map defined for $a\in RH$ by $\chi_{_1}(a)=\phi(1\otimes [\bar a] ),$ where $\phi$ is given by (\ref{phi}),
 and define  two subsets
  $\widetilde{\mathcal D},\,\widetilde{\mathcal D}_{\Bbbk}\subset {\mathcal D}_{\Bbbk}$  as
 \[\widetilde{\mathcal D}=\{a\in {\mathcal D}_{\Bbbk}\mid  \chi_{_1}(a)=0 \}
 \ \  \text{and} \ \
 \widetilde{\mathcal D}_{\Bbbk}=\{a\in {\mathcal D}_{\Bbbk}\mid  \chi_{_1}(a)=0\!\!\mod \mu \}
  \]
  (e.g. $\widetilde{\mathcal D}$ contains the expressions of the form $ab-(-1)^{(|a|+1)(|b|+1)}ba$  and also of the form
 $y^{\lambda}$ with $|y|$ odd and $\lambda$ even, while $\widetilde{\mathcal D}_{\Bbbk}$ contains $y^{\lambda}$ with $|y|$ even and
$\lambda$  divisible  by
$\mu\geq 2$).
Given $a\in RH,$ obviously $d_ha\in \widetilde{\mathcal D}$ implies  $\chi_{_{1}}(a)\in \operatorname{Ker} d_{\omega},$ while
$d_ha\in \widetilde{\mathcal D}_{\Bbbk}$ implies
     $d_{\omega}\chi_{_1}(a)=0 \mod \mu.$

The following statement is also simple.
\begin{proposition}\label{nonweak2}
 Given   $a\in RH,$ let  $d_ha\in \widetilde{\mathcal D}.$ If  $a=v_1v_2+c$ such that $c$ does not contain $\pm v_2v_1$ as a summand component   and
  $v_1v_2$ does not occur as a summand component of $d_hw$ for any  $w\in RH$
unless $w\in {\mathcal E}$ or $d_hw$  has a summand component from $V,$ too, then $[\chi_{_1}(a)]\neq 0$ in $H(RH\otimes \bar V,d_{\omega}).$
\end{proposition}
For $v_1=1,$ $v_2\in V$ and $c=0,$  taking into account (\ref{short})  the  proposition in particular implies that if $v_2$ is not $\lambda$-homologous to zero, then $[\chi_{_{1}}(v_2)]=[1\otimes \bar v_2]\neq 0$ in $H(RH\otimes \bar V,d_{\omega}).$

 By the universal coefficient theorem we have an isomorphism
\[H^n_{\Bbbk}\approx H^n\! \otimes \Bbbk \, \bigoplus\,  Tor(H^{n+1},\Bbbk):=H^n_{\Bbbk, 0}\bigoplus H^{n}_{\Bbbk,1}. \]
Given $\mu\geq 2,$ define a subset $K_{\mu}\subset V^{-1,*}$  as
\[
 K_{\mu}=\left\{  a\in n\!\cdot\!{\mathcal V}^{-1,*}\,|\, da=\lambda b\neq 0,\, b\in R^0H^* ,\,
\mu\ \text{divides}\ \lambda,\, 1\leq n<\mu  \right\}
\]
(i.e., $n=1$ for $\mu$ to be a prime)
and let  \[\Xi_{\mu}= K_{\mu}\cup {\mathcal V}^{0,*}.\]
Let $\mathcal{H}_{\Bbbk}  \subset H_{\Bbbk}$  denote a (minimal) set of multiplicative generators.
Given  $x\in {H}_{\Bbbk},$   let $x_0$ be  its
\emph{representative}   in $ RH$  with  $[t_{\Bbbk}x_0]=x;$
 in particular, $x_0\in R^0H^* $ for $x\in {H}_{\Bbbk,0},$
while
$x_0\in \mathcal{V}^{0,*} $  or
   $x_0\in K_{\mu}$ when $x\in \mathcal{H}_{\Bbbk, 0}$  or  $x\in \mathcal{H}_{\Bbbk, 1}$ respectively;
given
$x\in H_{\Bbbk, 1}$  and  $y\in H_{\Bbbk, 0}$
   with
 $dx_0=\lambda y_0$  and $\lambda$  divisible  by $\mu,$ we also denote such a connection between $x$ and $y$ as
\[\beta_{\lambda}(x)=y,\]
where $\beta_{\mu}$ is  the Bockstein cohomology homomorphism associated with the sequence $0\rightarrow {\mathbb Z}_{\mu}
\rightarrow {\mathbb Z}_{\mu^2}  \rightarrow {\mathbb Z}_{\mu} \rightarrow 0$ and is simply denoted by $\beta.$

On the other hand, if \[\sigma : H^*(A_{\Bbbk})\rightarrow H^{*-1}(BA_{\Bbbk}),\ \ [a]\rightarrow [\bar a] \]
is the cohomology suspension map,
 $x\in \operatorname{Ker}\sigma$ is equivalent  to say that $x_0$ is $\lambda$-homologous to zero in $(RH,d_h).$

Let $O_{\Bbbk}\subset R^0H$ be a subset given by
 \[
O_{\Bbbk} =\left\{  b\in  R^0H \,|\, da=\theta b\ \text{for}\  a\in V^{-1,*}\ \text{and}\
\theta\in {\mathbb Z}\ \text{is prime with}\ \mu  \right\}.
\]
Obviously $\rho_{\Bbbk}(O_{\Bbbk})=0$ and let ${\mathcal O}_{\Bbbk}\subset RH$ be a Hirsch ideal generated by $O_{\Bbbk}.$ In the sequel we consider the quotient Hirsch algebra $RH/{\mathcal O}_{\Bbbk}$
which
by abusing the notations we will   again denote by $RH.$

 Given  $w\in RH,$  let  $d_hw$ admit a decomposition
 \begin{equation}\label{split}
 d_hw=w_1+w_2,\ \   w_1\in
 \widetilde{\mathcal D}_{\Bbbk}, \       w_2=P(z_1,...,z_q)\in {\mathcal D}_{\Bbbk}\ \   \text{for} \ \  z_i\in \Xi_{\mu},
1\leq i\leq q.
 \end{equation}
Given $n\geq 0$ and assuming $w$ to be odd dimensional, define an indecomposable  element $x'(n)\in RH$ as
\begin{equation*}
x'(n)\!=\!
\left\{
\!
\begin{array}{llll}
{w}^{\smallsmile_1 (n+1)}, &   w_2=0,
\\
 \frac{\mu}{p}{w}^{\smallsmile_1 (n+1)} \smallsmile_1 Z_{w} +   \gamma_w , &   w_2\neq 0, & & \mu\geq 2,
\\
{w}^{\smallsmile_1 (n+1)} \smallsmile_1 {z_1}
 \cdots      \smallsmile_1 {z_q}             +\gamma_w, & w_2\neq 0, &|z_j|\ \text{is even}, &  \mu=0,
\end{array}
\right.
\end{equation*}
where
$Z_{w}=
 {z_1}^{ {\smallsmile_1}^{( p^{\nu_1+1}-1)}  }
 \smallsmile_1
 \cdots      \smallsmile_1 {z_q}^{ {\smallsmile_1}^{(p^{\nu_q+1}-1)}}$
with the convention that the component
$ {z_j}^{ {\smallsmile_1}^{( p^{\nu_j+1}-1)}  } $ is eliminated whenever  $[z_j]=\mathcal{P}_1([z_i])$ for  some
$1\leq i<j\leq q ,$ while
 $\gamma_w$ is defined so that
 $d_h  x'(n)
\in \widetilde{\mathcal D}_{\Bbbk} ;$
 namely, the existence  of  $\gamma_w$  uses
 Hirsch formulas (\ref{hirsch1})--(\ref{hirsch2})
and  the fact that
 $\frac{\mu}{p} {z_j}^{{\smallsmile_1}^{p^{\nu_j+1}}}$ (and $z_j\smile_1 z_j$  for $|z_j|$ even and $\mu=0$)
is $\!\!\mod \mu$ cohomologous to zero for all $j.$

Then  $w$ rises to
 the sequence
 $\mathbf{x}=\left\{x(n)\right\}_{n\geq 0}$ in $RH\otimes \bar V$
 defined by
 \begin{equation}\label{even}
x(n)=\chi_{_1}\!\left( x'(n) \right).
 \end{equation}
  Thus $d_{\omega} x(n)=0\!\! \mod \mu $ for all $n.$

 In the sequel we apply to (\ref{even}) for the following
specific cases of $w.$ First,
given   $x\in {\mathcal H}_{\Bbbk}$ and  its representative $x_0\in \Xi_{\mu}\subset RH,$ the element  $w:=x_0$ obviously satisfies (\ref{split})
(with $w_2=0$); thus for $x\in {\mathcal H}^{od}_{\Bbbk},$ equality (\ref{even}) is specified as
 \[x(n)=1\otimes s^{-1}\left(  {x_0}^{\smallsmile_1 (n+1)}\right). \]
\begin{example}\label{example}
Let $\mu=2$ and   $x\in {\mathcal H}^{od}_{\Bbbk}.$ When ${\mathcal P}_1(x)=0,$   we have a relation  in $(RH,d_h)$
\[d_hv=x_0\smile_1x_0+2x_1\ \ \text{with}\ \ dx_1=x^2_0,\ \ \    x_1\in V^{-1,*},\,v\in V^{-2,*}.   \]
 Therefore, $x_{\mu}(1)=\left\{
\begin{array}{lll}
1\otimes \overline{x_0\smile_1 x_0}, & {\mathcal P}_1(x)\neq 0,\, \, \, |x|\   \text{is the smallest},  \\
1\otimes \bar x_1, &  {\mathcal P}_1(x)=0
\end{array}
\right.
$
 in $RH\otimes \bar V.$
\end{example}

Furthermore,   $x\in {\mathcal H}^{od}_{\Bbbk}$    rises to the sequence $\{x_n\in V\}_{n\geq 0}$  in
$(RH,d):$

For $x\in {\mathcal H}_{\Bbbk,0}$   ($x_0\in {\mathcal V}^{0,*}$),
\begin{equation}\label{msymmetric}
dx_n=\sum_{\substack{i+j=n-1\\i,j\geq 0}}\varepsilon_{i,j}x_ix_j,\ \ \ \ \ \
\varepsilon_{i,j}=\left\{\!\!\!
\begin{array}{lll}
2, &   \rho x_0^2\neq 0,  &  i,j\ \text{are even},\\
1,& \text{otherwise},
\end{array}
\right.
\ \ \ \ n\geq1,
\end{equation}
with $x_1=-x_0\smile_1x_0$ when $\rho x_0^2\neq 0.$

 For $x\in {\mathcal H}_{\Bbbk,1}$ with  $dx_0=\lambda x'_0$ ($x_0\in K_{\mu}$),
\begin{equation}\label{msymmetric2}
dx_n=\sum_{\substack{i+j=n-1\\i,j\geq 0}}x_ix_j+\lambda x'_n,\ \ \ \ \
dx'_{n}=-\frac{1}{\lambda}\,d\left( \sum_{\substack{i+j=n-1\\i,j\geq 0}}x_ix_j \right),\ \ \ n\geq1.
\end{equation}
 The element $x_n$ is of odd
degree   in (\ref{msymmetric})--(\ref{msymmetric2}) for all $n.$
The action of $h$ on $x_n$ in $(RH,d_h)$  is given by the following formula.
Let  $\tilde x_i=y_i+h^{tr}x_i $  with  $y_i=0$ or $dy_i=-\lambda h^{tr}x'_i$
for $x_i$ in (\ref{msymmetric}) or (\ref{msymmetric2}) respectively;
  for
 $0\leq i_1\leq \cdots \leq i_{r}<n,\,r\geq2,$  denote also
    $\tilde{x}_{i_1,...,i_r}=(-1)^{r}\tilde{x}_{i_1}\cup_2\cdots\cup_2\tilde{x}_{i_r}.$
   \begin{equation}
\label{hmsymmetric}
hx_n=\tilde{x}_n+
\sum_{\substack{i_1+\cdots +i_r+r=n+1}}\varepsilon_{i_1,...,i_r}\left(\tilde{x}_{i_1,...,i_{r-1}}\smile_1 x_{i_r}-
\tilde{x}_{i_1,...,i_r}\right),
 \end{equation}
 \[
  \hspace{0.47in}\varepsilon_{i_1,...,i_r}=\left\{\!\!\!
\begin{array}{lllll}
2, & x\in H_{\Bbbk,0},\,  \rho x_0^2\neq 0, \  \text{some}\ (i_s,i_t)_{1\leq s,t\leq r}\ \text{is even},\\
1,& \text{otherwise}.
\end{array}
\right.
\]

Now given $x\in H^{od}_{\Bbbk}$  and the smallest odd prime $p$ that divides $\mu,$  recall the definition of  the  symmetric Massey product $\langle x\rangle ^{p}$ \cite{kraines} for which  we have
  the Kraines formula (see also \cite{saneFiltered})
\begin{equation}\label{krainesf}
   \beta {\mathcal P}_1(x)=-\langle x\rangle ^{p}.
\end{equation}
We also have the equality
\begin{equation}\label{comassey}
\rho_{\Bbbk} h^{tr}(x_{p-1})=-\langle x\rangle ^{p}.
\end{equation}
When  ${\mathcal P}_1(x)\in H_{\Bbbk,0},$  we have   $\beta {\mathcal P}_1(x)=0= \langle x\rangle ^{p}.$
Hence we obtain
$\rho_{\Bbbk}h^{tr}(x_{p-1})=0,$ and, consequently,  $h^{tr}x_{p-1}=0.$ Then
$dx_{p-1}\in \widetilde{\mathcal D}_{\Bbbk}$  implies that the element
\begin{equation}\label{mzero}
w=x_{p-1}
\end{equation}
 satisfies $(\ref{split}).$

\subsection{Element $\varpi\in V$ associated with a relation in $H_{\Bbbk}$}\label{w-associated}
Given any two multiplicative generators $a,b\in H_{\Bbbk,1}$ with $da_0=\lambda_aa'_0$ and $db_0=\lambda_b b'_0,$ $a'_0,b'_0\in {V}^{0,*},\, a_0,b_0\in K_{\mu},$ we have a relation in $(RH,d)$
\begin{multline}\label{ab}
 du=a_0b_0+\lambda u'\ \ \text{with}\ \  du'= - \frac{\lambda_a}{\lambda}a'_0b_0 - (-1)^{|a|}\frac{\lambda_b}{\lambda} a_0b'_0, \\
 \lambda=\operatorname{g.c.d.}(\lambda_a,\lambda_b),\ \  u\in  V^{-3,*},\, u'\in V^{-2,*}.
 \end{multline}
Regarding the action of $h$   in $(RH,d_h),$ we have   $hu=(h^2+h^3)u$ and, in particular,
the relation $ab=0$ in $H_{\Bbbk}$
is equivalent to
the equalities $h^2u=0$ and   $h^3u=0 \, \mod \mu$ in $(RH,d_h).$ More generally,
       a relation of the form $ab+c_1+c_2=0$ in $H_{\Bbbk},$  where
 $c_1=P_1(a_1,...,a_{q_1})=\sum_r \lambda_{r} a_{1,r}^{n_{1,r}}\cdots a_{p_r,r}^{n_{q_r,r}}$
 with  a single   $a_{i,r}\in H_{\Bbbk,1}$ with $n_{i,r}=1$ and the other   $a_{j,r}\in H_{\Bbbk,0}$  for each $r,$   and $c_2=P_2(b_1,...,b_{q_2})$ with   $b_j\in H_{\Bbbk,0}$ for all $j,$   yields  that
$h^2u\in {\mathcal D}_{\Bbbk}$ and  $h^3u\in {\mathcal D}_{\Bbbk}.$
If either $a$ or $b$ is from $H_{\Bbbk,0},$ then to $ab=0$ corresponds
the equality given by a similar formula  as  (\ref{ab}) but this time   $(u,u')\in  (V^{-2,*},  V^{-1,*})$ with  $du'= - (-1)^{|a|}a_0b'_0$
or $ du'= - a'_0b_0$ respectively.
Since $a_0b_0$  is $\!\!\!\mod \mu$  cohomologous to $hu,$
we have  that for a given  $1\leq k< r,$    any  cocycle in  $\bigoplus_{0\leq i\leq r} R^{-i}H$  is $\!\!\!\mod \mu$  cohomologous to a  cocycle in $\bigoplus_{0\leq i\leq k} R^{-i}H.$

In general, consider a (homogeneous) multiplicative relation in $H^*_{\Bbbk}$
 \begin{equation}\label{relation}
P(y_1,...,y_q)=0,\ \  \ y_i\in \mathcal{H}_{\Bbbk} \end{equation}
which is
 not a consequence of the commutativity of the algebra $H_{\Bbbk}$ (and also is not decomposable by any other relations).
 Obviously
 $P(b_1,\!...,b_q)$ for $b_i=(y_i)_0\in \Xi_{\mu}$ is a $\!\!\!\mod \mu$ cohomologous to zero cocycle  in $(RH,d_h).$
If $P(b_1,\!...,b_q)\in \bigoplus_{0\leq i\leq r} R^{-i}H $ with $r\geq 3,$  then, as above,  there is a
  $\!\!\!\mod \mu$ cohomologous to $P(b_1,\!...,b_q)$  cocycle
$P'(z_1,...,z_{m})$ that lies in $\bigoplus_{0\leq i\leq 2} R^{-i}H;$ in particular
$z_i\in \Xi_{\mu}$ for all $i,$ but
   each monomial of $P'(z_1,...,z_{m})$ may contain at most two variables $z_i$ from $ K_{\mu}.$ So that
  we have
   one of the following
   equalities in $(RH,d_h):$
   \begin{equation}
    \label{basic}
  d_hu=  \left\{\!\!   \begin{array}{lllll}
P(b_1,\!...,b_q),     &  u\in V^{-1,*},\hspace{0.6in}   y_i\in {\mathcal H}_{\Bbbk,0},\,\,1\leq i\leq q, \\

           P'(z_1,...,z_{m}), &  u\in V^{-2,*},  \\

    P'(z_1,...,z_{m}) +\lambda u' , &  (u,u')\in (V^{-r,*}, V^{-r+1,*}),\ \ \ \ \ \ \ \, r=2,3.
     \end{array}
     \right.
    \end{equation}
Note  that $d_h u\notin \widetilde{\mathcal D}_{\Bbbk}$ (since $du\notin \widetilde{\mathcal D}_{\Bbbk}$), unless each monomial of $P(y_1,...,y_q)$
is of the form $\alpha_y y^{\lambda},\, \alpha_y\in {\mathbb Z}$ with $\lambda$  divisible by $\mu\geq 2$ in which case $d_h u\in \widetilde{\mathcal D}_{\Bbbk}.$ In any case
 $w:=u$  obviously satisfies (\ref{split}).
           For example, given  $y\in
\mathcal{H}_{\Bbbk}$  with $\dim H_{\Bbbk}<\infty,$ let $\hbar_y$
be the  height of $y$ with respect to the product on $H_{\Bbbk}$ so that we have the relation   $P(y)=y^{\hbar  _y+1}=0$  and then  (\ref{basic})
  becomes  the form
  \begin{equation}
    \label{power}
  d_hu=  \left\{\!\!  \begin{array}{lllll}
     y_0^{\hbar  _y+1}, &  y\in {\mathcal H}_{\Bbbk,0},\vspace{1mm}\\
     y'_0y_0+\lambda u',  &  y\in {\mathcal H}_{\Bbbk,1},
     \end{array}
     \right.
    \end{equation}
    where $y'_0\in V^{0,*}\oplus V^{-1,*}$ is $\!\!\!\mod \lambda$ cohomologous to $y_0^{\hbar_{y}}\in R^{-\hbar_{y}}H^*$ in $(RH,d_h)$
    (in particular, $y'_0=y_0$ for $\hbar_y=1;$ cf. (\ref{msymmetric2})).
  \begin{remark}
 Note that when  both $|y|$  and $\mu$ are odd,  always $\hbar_y=1$ and we say that (\ref{power}) is a consequence of the commutativity of $H_{\Bbbk}.$
\end{remark}
Now given (\ref{relation})  of the smallest degree  and   assuming
$\tilde {H}_{\mathbb{Q}}$
is either trivial or has a single algebra generator,
we define an odd dimensional $\varpi\in V$  in the following three cases.  In the case  $\tilde {H}_{\mathbb{Q}}\neq 0,$
denote a single multiplicative generator of infinite order of $H$ by ${\mathfrak z}$ and let $z= t^*_{\Bbbk}({\mathfrak z});$ thus $z={\mathfrak z}\otimes 1\in H_{\Bbbk,0}.$ (Warning: $z$ may not be  a multiplicative generator of $H_{\Bbbk}.$)

(i) When  $P$  is even dimensional in (\ref{relation}),  $u$ is  odd dimensional in (\ref{basic}) and we set $\varpi=u.$

(ii)  When  $P$   is odd dimensional in (\ref{relation}),  $u$ is even dimensional in (\ref{basic}) and we have to consider the following subcases.
Suppose that the following  expression
 \begin{equation}
\label{beta}
\beta_{\lambda}(P(y_1,...,y_q))=\sum _{1\leq i\leq q}(-1)^{|y_1|+\cdots +|y_{i-1}|}P(y_1,...,\beta_{\lambda}(y_i),...,y_q)
\end{equation}
is  formally  (i.e., independently on $H_{\Bbbk}$) trivial.

$(ii_1)$ Let  $y_i\in {\mathcal H}_{\Bbbk,0}$ for all $i,$ i.e., $\beta_{\lambda}(y_i)=0$  and the corresponding relation of (\ref{relation}) in $(RH,d_h)$ is given by the first equality of (\ref{basic}).
 Since either
at most one
$y_i$ may be equal to $z$ for $q\geq 2$ or $q=1$ and $y_1=z$ with $z\in {\mathcal H}^{od}_{\Bbbk,0}$ for $\mu$ even,
  it is easy to see that there is $c\in H_{\Bbbk}^{+}\cdot H_{\Bbbk}^{+}$ such that $\beta_{\lambda}(c),$ as a formal expression, is equal to $P(y_1,...,y_q).$ (In the last case $c=\lambda_z{\mathcal P}_1(z)z^{2m-1}$  for $P(z)=\lambda_z z^{2m+1},m\geq 1.$)
  This situation    answers to  the following relations in $(RH,d_h).$
There is a pair $(b_c,w_c)$ with $b_c\in {\mathcal D},$ $w_c\in V$  such that
 \[dw_c=-b_c+\lambda u \ \ \text{with}  \ \  db_c= \lambda P(b_1,...,b_{q}) \]
where  $[t_{\Bbbk}h^{tr}w_c]=c$ and $u$ is given by the first equality of (\ref{basic}).
(In particular $\lambda=\mu$ and $b_c=-\lambda_z \frac{\mu}{2}(z_0\smile_1 z_0) z_0^{2m-1}$  when $\mu$ is even  and $P(z)=\lambda_z z^{2m+1}$ as above.)
Therefore, if $c={\mathfrak c}\otimes 1$ with ${\mathfrak c}\in H,$ then ${\mathfrak c}=\rho h^{tr}w_c.$
Note also that   ${\mathfrak c}  $ must be indecomposable in $H$ since there is  no relation in degrees $<|P|$ in $H_{\Bbbk}.$
When ${\mathfrak c}$ is of finite order,
 there is $a\in H_{\Bbbk,1}$ with $\beta_{\lambda}(a)=c$ so that   $da_0=\lambda  h^{tr}w_c  .$
Furthermore,
when  itself  $a$  is linearly dependent on  $\mathcal{P}_1(x)$ for some $x\in H^{od}_{\Bbbk},$ i.e.,
$k a=\ell\, \mathcal{P}_1(x),$  $k,\ell \in {\mathbb Z},$   we obtain
 $ \frac{k \lambda}{\mu}\, c=- \ell\, \langle x\rangle ^{p}$ since   (\ref{krainesf}).
Taking into account  (\ref{comassey})   we have that
  $ \frac{k \lambda}{\mu}\,h^{tr}w_c$ is $\!\!\mod \mu$ cohomologous  to $ \ell\, h^{tr}x_{p-1},$  i.e., there is $v\in V^{-1,*}$ with $dv=
 \frac{k \lambda}{\mu}\,h^{tr}w_c-\ell\, h^{tr}x_{p-1}+\mu v',\,v' \in V^{0,*}.$
Define   $\varpi\in V $  by
\begin{equation}\label{ww}
  \varpi=\left\{
    \begin{array}{llll}
     \frac{k \lambda}{\mu}\,w_c- v-   \ell\, x_{p-\!1} , & k a= \ell\, {\mathcal P}_1(x)\neq 0,  \\
    a_0  , & \hbox{otherwise.}
    \end{array}
  \right.
\end{equation}
Obviously $w:=\varpi$  satisfies (\ref{split}).
Let now  ${\mathfrak c}$ be of infinite order. Since ${\mathfrak z}$ is unique (when it exists),  neither $y_i$ occurs as $z,$ so
there is the other $\bar c\in H_{\Bbbk}$ such that $\beta_{\bar \lambda}(\bar c),$ as a formal expression, is again equal to $P(y_1,...,y_q).$
This time
 $\bar{\mathfrak c}$ is of finite order and hence
the definition of $w$  by means of   formula (\ref{ww}) is not obstructed.

$(ii_2)$ Let at least two $y_i$ lie in ${\mathcal H}_{\Bbbk,1}.$ Then
 the corresponding relation of (\ref{relation}) in $(RH,d_h)$     is given by the second equality of (\ref{basic}).
Since
for each $y_i\in {\mathcal H}_{\Bbbk,1}$ with  $\beta_{\lambda}(y_i)=y'_i$ either $y'_i\in {\mathcal H}_{\Bbbk,0}$ or $y'_i=z^n $ for
$z\in {\mathcal H}_{\Bbbk,0},$
$n\geq2,$
  there exist $c\in H_{\Bbbk}^{+}\cdot H_{\Bbbk}^{+}$ such that $\beta_{\lambda}(c),$ as a formal expression, is equal to $P(y_1,...,y_q).$
When ${\mathfrak c}$ is of finite order, we can define $w$ entirely analogously as in item (i), i.e., by formula (\ref{ww}); otherwise, for ${\mathfrak c}$ to be of infinite order, we get an obstruction, i.e., when we have
\begin{equation}\label{specific2}
\beta_{\lambda}(c)=P(y_1,...,y_q)=0\ \   \text{for}   \ \ c=z\ \  \text{modulo decomposables} .
\end{equation}
Note also that $du\notin \widetilde{\mathcal D}_{\Bbbk}$ for $u$ corresponding to this relation by (\ref{basic}).

(iii) Suppose that   expression (\ref{beta})
 is not formally trivial.
Then the corresponding relation in $(RH,d_h)$ is given by the third equality of  $(\ref{basic}).$ Consider two subcases.

$(iii_1)$ Let (\ref{relation})
 be specified as
\begin{multline}\label{apowerb}
P(y_1,..,y_q)=\lambda'  b^nc\,x=0\ \, \text{with}\ \  \beta_{\lambda}(b)=c \ \      \text{for}\\
       b\in {\mathcal H}^{ev}_{\Bbbk,1},\,c\in {\mathcal H}^{od}_{\Bbbk,0},\,x\in H_{\Bbbk,0},\, \lambda'\in{\mathbb Z} ,    \, n\geq 1.
  \end{multline}
In particular  $\beta_{\lambda}(\lambda' b^{k}x)=0$ for $k>n.$ Since $|b^{n+1}|<|b^nc|$ and (\ref{apowerb}) is chosen to be of the smallest degree, we have  $\lambda' b^{n+1}x\neq 0.$
Consider two elements $a_i$ for $i=n+1,n+2$ with $\beta_{\lambda}(a_i)=0$ where
\begin{equation*}
  a_i=\left\{
    \begin{array}{llll}
  b^i , & i\ \text{is divisible by}\ \mu  \\
     \lambda' b^ix  , & \hbox{otherwise}.
    \end{array}
  \right.
\end{equation*}
 When   $ a_{n+2}\neq 0,$   there is $a\in \mathcal{H}^{od}_{\Bbbk,1}$
with  $\beta_{\lambda}(a)=a_{n+1}$
or $\beta_{\lambda}(a)=a_{n+2}$ and  we set $\varpi=a_0.$ When $a_{n+2}=0,$  we  consider this relation as a particular case of (\ref{relation})  and  set $\varpi=u$ as in item (i) above.

$(iii_2)$
When at least one monomial of (\ref{relation})  differs from that given by (\ref{apowerb}),
 we have  $\beta_{\lambda}(P(y_1,...,y_q))=0$
is a desired relation, and
then set $\varpi=u'$ where $u'$ is resolved from (\ref{basic}).

Finally,
we say that an odd dimensional element $\varpi\in V$ is \emph{associated} with (\ref{relation})   if
$\varpi$ is given by one of items $(i)$--$(iii)$ above. In particular, $\varpi$ always exists for $\tilde{H}_{\mathbb Q}=0$
or, more generally, for $z\in {\mathcal H}_{\Bbbk}.$

Given
an even dimensional $y\in {\mathcal H}^{ev}_{\Bbbk,0}$ with the relation
\begin{equation}\label{evenz}
P(y)=\lambda_{y} y^{m}=0,\ \ m\geq 2,\,\lambda_y\in {\mathbb Z},
\end{equation}
it rises to  the sequence $\{y_n\in V\}_{n\geq 0}$  in
$(RH,d):$
\begin{equation}\label{m-even}
\begin{array}{rllll}
 d y_{2k+1}&=&
\underset{{\substack{i+j=k-1\\i,j\geq 0}}}{\sum} y_{2i+1}y_{2j+1}-\,\,
  \underset{{\substack{i_1+\cdots +i_m=k\\
 0\leq i_1\leq...\leq i_m\leq k}}}{\sum}\lambda_y y_{2i_1}\!\cdots y_{2i_m},
             \vspace{5mm}\\
d y_{2k} &=&\underset{ {\substack{i+j=2k-1\\i,j\geq 0}}}{\sum}(-1)^{i+1}y_{i}y_{j}
  ,\hspace{1.9in} k\geq 0,
\end{array}
\end{equation}
where   $y_n$ is of odd degree   for $n=2k+1$ and is of even degree for $n=2k.$
In fact a straightforward check shows that each $y_{2k},k\geq 1,$
  can be
expressed in terms of    $y_{r} $  for $r<2k$  as $y_{2k}=-y_0\smile_1 y_{2k-1}\!\!\mod {\mathcal D}$
(e.g. $y_2=-y_0\smile_1y_{1}+\lambda_z\!\underset{i+j=n-1}{\sum} y_0^{i}(y_0\cup_2 y_0)y_0^{j}$\,).
Consequently,
      $h^{tr}(y_{2k})\in {\mathcal D}.$
\begin{proposition}\label{norelation}
Let $x\in {\mathcal H}^{od}_{\Bbbk}$ and let (\ref{evenz}) be a single relation in $H_{\Bbbk}$ with $|P|< |x|.$
Then $x\notin \operatorname{Ker} \sigma$
 and if there is $b\in H_{\Bbbk,1}$ with $\beta_{\lambda}(b)=x,$  then also $b\notin \operatorname{Ker} \sigma.$

\end{proposition}
\begin{proof}
Theoretically there may be $y_n\in V$ given by (\ref{m-even}) serving as a source for $h$ to kill $x$ or  $ b.$
 Since $|x|$  is odd and $|h^{tr}y_{2k+1}|$ is even  (and $h^{tr}(y_{2k})\in {\mathcal D}$), $x_0$ is not $\lambda$-homologous to zero.
 When there is $b\in H_{\Bbbk,1}$ with  $\beta_{\lambda}(b)=x,$   the element
 $b_0\in K_{\mu}$ is not $\lambda$-homologous to zero since (\ref{m-even}) and $d^2_h=0$ prevent $b_0$ to be in the target of $h$ evaluated on any $y_{2k+1}.$
\end{proof}

\begin{proposition}\label{sequences}
Let $x\in {\mathcal H}^{od}_{\Bbbk}$  and let  $P(y_1,...,y_q)=0$
 be a relation  given by $(\ref{relation}).$

  $(i)$ If $|x|\leq |P|$
or (\ref{relation}) is specified as (\ref{evenz})  to be
 a  single relation with   $|P|<|x|,$
      then
 the
sequence  $\mathbf{x}=\left\{x(i)\right\}_{i\geq 0}$ given by (\ref{even}) for $w=x_0$ satisfies (\ref{p1})--(\ref{p2})  in $RH\otimes\bar V.$

$(ii)$ Let $\varpi\in V$ be associated with the relation $P(y_1,...,y_q)=0.$

$(ii_1)$ If $P(y_1,...y_q)$ is  of the smallest degree with  $|x|<|P|,$
then the
         sequence  $\mathbf{x}=\left\{x(i)\right\}_{i\geq 0}$ given by (\ref{even}) for  $w=\varpi$  satisfies (\ref{p1})--(\ref{p2})
         in $RH\otimes\bar V.$

$(ii_2)$ Let  $P'(y'_1,...,y'_{q'})=0$ and $P(y_1,...y_q)=0$ be two relations of the smallest degree   with  $|x|<|P'|\leq |P|$ where
the first relation is
 given by (\ref{specific2}).
       Then the
         sequence  $\mathbf{x}=\left\{x(i)\right\}_{i\geq 0}$ given by (\ref{even}) for  $w=\varpi$  satisfies (\ref{p1})--(\ref{p2})
         in $RH\otimes\bar V.$

$(iii)$
 Let    $P(y_1,...,y_q)\neq  \lambda_x x^2$
    for $\mu$   even and $\mathcal{P}_1(x)=0,$ $\lambda_x\in{\mathbb Z}.$  Then
the pair of sequences  given by  items  (i) and (ii)
  is admissible.
\end{proposition}
\begin{proof}
(i)  First note that when $|x|\leq |P|,$ $x_0$ is not $\lambda$-homologous to zero by the degree reason, while apply Proposition \ref{norelation} for  $|P|<|x|;$ the same argument implies that $b_0$ is also not $\lambda$-homologous to zero when $\beta_{\lambda}(b)=x.$
Consider two subcases.

$(i_a)$ Let $x\in H_{\Bbbk,0}.$ In fact we have   to verify only (\ref{p1}).
 $(i_{a1})$ Assume there is no $b$ with $\beta_{\lambda}(b)=x.$
 Observe that for an odd dimensional  $a\in RH$ and $n\geq 2,$
        $da^{\smile_{\mathbf{1}}n}$    contains a summand component of the form
\begin{equation}\label{binomial}
-\sum_{k+\ell=n}\binom{n}{k} a^{\smile_{1}k} a^{\smile_{1}\ell},\,\,\,\,k,\ell\geq1\ \ \ (\text{with}\ \  a^{\smile_1 1}=a).
\end{equation}
By setting  $v={x_{0}}^{\smile_{1}n}$ and $v_1v_2=-\binom{n}{k} {x_0}^{\smile_{1}k} {x_0}^{\smile_{1}\ell},$ some $k,\ell,$ we see that the hypothesis of Proposition
\ref{nonweak} is satisfied
and hence  ${x_{0}}^{\smile_{1}n}$ is not weakly homologous to zero. Consequently,  $[x(n)]\neq0$ as desired.
$(i_{a2})$
 Assume there is $b$ with $\beta_{\lambda}(b)=x.$ Let
  $p$ be a prime that divides $\mu.$
 We have a sequence of relations in $(RH,d)$
\begin{equation}\label{oneone}
 d b_{n}=
 \sum_{\substack{i+j=n\\i\geq0;j\geq 1}}\varepsilon_n \binom{n+1}{i+1} b_{i}{x_{0}}^{\smile_{1}j}+ \varepsilon_n \lambda\, {x_{0}}^{\smile_{1}(n+1)},\ \ \ b_n\in V,\,\,n\geq1,
 \end{equation}
where $\varepsilon_{p^k-1}=\frac{1}{p}$ and  $\varepsilon_n=1$ for  $n+1\neq p^k,k\geq 1. $
In view of (\ref{oneone}) and Proposition \ref{nonweak} we remark that  ${x_{0}}^{\smile_{1}n}$ may be weak homologous  to zero only  for $\lambda =\mu =p=n+1.$ In any case consider the element $a_n\in (RH,d_h)$ given by
\[   a_n=\sum_{\substack{i+j=n\\i\geq0;j\geq 1}}\varepsilon_n \binom{n+1}{i+1} b_{i}{x_{0}}^{\smile_{1}j}+hb_n. \]
Since $b_0$ and $x_0$ are not $\lambda$-homologous to zero,  so is $b_n$ for all $n.$
Obviously $d_ha_n\in \widetilde{\mathcal D}$ and by setting $a=a_n$ and    $v_1\cdot v_2=\varepsilon_n (n+1) \,b_{n-1}\cdot x_0,$ the component of $a$ for $(i,j)=(n-1,1),$ the hypotheses of Proposition \ref{nonweak2} are satisfied.
Therefore, we get
 $[\chi_{_{1}}(a)]\neq 0$ in $H(RH\otimes \bar V,d_{\omega}).$
Obviously  $[\chi_{_{1}}(a)]= -\varepsilon_n\lambda [\chi_{_{1}}({x_{0}}^{\smile_{1}(n+1)})].$ Thus  $[x(n)]\neq 0$  as desired.

$(i_{b})$ Let $x\in H_{\Bbbk,1}.$ Then we have to verify only (\ref{p2}).
Since $x_0\in K_{\mu}$ is not $\lambda$-homologous to zero,
the proof easily follows from the analysis of the component given by
(\ref{binomial}) for $a=x_0$ in
 $d{x_{0}}^{\smile_{1}n}.$

(ii)
When  $\varpi$    is  not $\lambda$-homologous to zero and either $d_h \varpi \in \widetilde{\mathcal D}_{\Bbbk}$ or
$d_h \varpi \notin \widetilde{\mathcal D}_{\Bbbk}$ but
 ${z_j}^{{\smallsmile_1}^{(   p^{\nu_j+1}-1)}}$    is also not $\lambda$-homologous to zero for all $j$ ($z_j$ is a variable  in $d_ h \varpi$),
 the proof is  analogous to that of  subcase $(i_{a1}) $ or $ (i_{b})$ of item (i).
Otherwise, we observe that  $\varpi$    is again not $\lambda$-homologous to zero in $(ii_1),$  while
$\varpi$ may be   $\lambda$-homologous to zero  in  $(ii_2)$
only by evaluating $h$ on certain elements  $u_i\in V$ arising from the relation given by (\ref{specific2}) the first of which is
$u=u_0$  as given by (\ref{basic});
since $d u \notin \widetilde{\mathcal D}_{\Bbbk},$ neither $du_i$ is in $\widetilde{\mathcal D}_{\Bbbk}.$
And then in the both subcases a straightforward check
 completes the proof.

(iii) The proof is analogous to that of items $(i)$--$(ii).$  The restriction on the relation for $\mu$ even is in fact explained by Example \ref{example}.

 \end{proof}

 \begin{remark} Let $A_{\Bbbk}=C^*(X;{\mathbb Z}_p),p>2.$   The case of $x\in H_{{\mathbb Z}_p}$ with $\beta(b)=x$  fundamentally distinguishes the (based) loop  and free loop spaces on $X$ with respect to the existence of infinite sequence arising from $x$ in
 $H^*(\Omega X;{\mathbb Z}_{p})$ and $H^*(\Lambda X;{\mathbb Z}_{p})$ respectively.
  Namely,
  let both
    $\mathcal{P}_{1}(x)$ and  $\langle x \rangle^p$  be  multiplicative generators of  $H^*(X; {\mathbb Z}_p)$
 such that
 $\mathcal{P}_{1}(x)\in\langle b,x,...,x \rangle,$ the $p^{th}$-order Massey product.
 Then by the hypotheses of Proposition \ref{sequences}
 the sequence arising from $x$ in
 $H^*(\Omega X;{\mathbb Z}_{p})$ may terminate at the $p^{th}$-component $($see \cite{saneFiltered} for $p=3)$, while  is always infinite in $H^*(\Lambda X;{\mathbb Z}_{p}).$
 \end{remark}

\section{Proof of Theorem \ref{varsigma}} \label{theorem}

The proof of the theorem relies on the two basic propositions below in which the
condition that $\tilde{H}_{\Bbbk}$ requires at least two algebra generators is treated in two
specific cases. Note also that the essence of the method used in the proof of the following proposition  is in fact kept   for $\mu$ to be a prime.

\begin{proposition}
\label{one} Let $H_{\Bbbk}$ be a finitely generated $\Bbbk$-module with $\mu\geq 2.$ If
$\tilde{H}_{\Bbbk}$ requires at least two algebra generators and $\tilde {H}_{\mathbb{Q}}$
is either trivial or has a single algebra generator, then
 there are two sequences
$\mathbf{x}_{\mu}=\left\{x_{\mu}(i) \right\}_{i\geq0}$ and
$\mathbf{y}_{\mu}=\left\{y_{\mu}(j)\right\}_{j\geq 0}$ of $\mod \mu$ $d_{\omega}$-cocycles in
$(RH\otimes \bar V,d_{\omega})$ whose degrees form arithmetic progressions and
the product classes $\left\{[t_{\Bbbk}x_{\mu}(i)]\cdot [t_{\Bbbk}y_{\mu}(j)]\right\} _{i,j\geq0}$ are linearly
independent in $H(RH\otimes \bar V _{\Bbbk},d_{\omega}).$
\end{proposition}

\begin{proof}

First,  we exhibit two sequences
$\mathbf{x}=\left\{x(i)\right\}_{i\geq 0}$ and $\mathbf{y}'=\left\{y'(i)\right\}_{i\geq
0}$ in $(RH\otimes \bar V,d_{\omega})$ consisting of $\!\!\!\mod \mu$ $d_{\omega}$-cocycles  and
satisfying the hypotheses of Proposition \ref{sequences}.
 In the case  $\tilde {H}_{\mathbb{Q}}\neq 0,$ let
${\mathfrak z}$ be a  single multiplicative generator of infinite order of $H$ and $z= t^*_{\Bbbk}(\mathfrak z)\in H_{\Bbbk,0}.$

By the hypotheses of the proposition  there is an odd dimensional element  $x\in \mathcal{H}_{\Bbbk}$
and choose
   $x$ to be of the smallest degree.
 Define $\mathbf{x}=\{x(i)\}_{i\geq0}$
   by (\ref{even}) for $w=x_0.$

   To find the second sequence, note that
   there must be  an even dimensional element  $y\in \mathcal{H}_{\Bbbk}$  and hence a relation $y^{\hbar_y+1}=0 $ in $H_{\Bbbk}$
   unless maybe $\mu$ is even and $x=z$ with ${\mathcal P}_1(x)\neq 0$ in which case we have  a relation $x^{\hbar_x+1}=0$ instead.
 First, observe the following: if there is $y\in {\mathcal H}^{od}_{\Bbbk}$ with  $y\notin \operatorname{Ker}\sigma$ and  linearly independent with ${\mathcal P}_1^{(m)}(x)$  for some $m,$
define    $\mathbf{y}'=\{y'(j)\}_{j\geq0}$
   by (\ref{even}) for $w=y_0;$
If there is $y\in {\mathcal H}^{od}_{\Bbbk,0}$ linearly dependent on ${\mathcal P}_1^{(m)}(x),$  while  ${\mathcal P}_1^{(m-1)}(x)\notin \operatorname{Ker}\sigma ,$
define     $\mathbf{y}'=\{y'(j)\}_{j\geq0}$
   by (\ref{even}) for $w$  given by  (\ref{mzero}) in which $x$ is replaced by  ${\mathcal P}_1^{(m-1)}(x).$
Otherwise, consider a relation $P(y_1,...,y_q)=0$  of the smallest degree in $H_{\Bbbk}$
     unless
    $P(y_1,...,y_q)= \lambda_x x^2,\lambda_x\in{\mathbb Z},$
    whenever  $\mu$ is odd or $\mu$ is even and $\mathcal{P}_1(x)=0.$
    When the relation admits to associate $\varpi$ as in subsection \ref{w-associated}, define    $\mathbf{y}'=\{y'(j)\}_{j\geq0}$
   by (\ref{even}) for  $w=\varpi.$ When the definition of $\varpi$ is obstructed, consider the next relation in $H_{\Bbbk}.$ This time
the second relation admits
to associate $\varpi$ (since the above ${\mathfrak z}$ is unique)
  and hence
  the second sequence  $\mathbf{y}'=\{y'(j)\}_{j\geq0}$
   is defined.

Finally,  the pair of the sequences
 $(\mathbf{x},\mathbf{y}')$ found above is admissible:
when the existence of the pair involves  relation(s) in $H_{\Bbbk}$ (if not, the claim is rather obvious), it
satisfies  the hypotheses of
 Proposition \ref{sequences}. Obtain
 the associated sequences
  $\mathbf{x}_{\mu}=\left\{x_{\mu}(i) \right\}_{i\geq0}$ and
$\mathbf{y}_{\mu}=\left\{y_{\mu}(j)\right\}_{j\geq 0}$ as in subsection
\ref{twosequences}. The explicit product on $RH\otimes \bar V$ allows us to ensure
immediately that the product classes $\left\{[t_{\Bbbk}x_{\mu}(i)]\cdot [t_{\Bbbk}y_{\mu}(j)]\right\}
_{i,j\geq0}$ are linearly independent in $H(RH\otimes \bar V _{\Bbbk},d_{\omega}).$
\end{proof}
Given a cochain complex $(C^{\ast},d)$ over $\mathbb{Q},$ let
$S_{C}(T)=\sum_{n\geq0}(\dim_{\mathbb{Q}}C^{n})T^{n}$ and $S_{H(C)}
(T)=\sum_{n\geq0}(\dim_{\mathbb{Q}}H^{n}(C))T^{n}$ be the Poincar\'{e} series. Recall
the convention:
 $\sum_{n\geq0}a_{n}T^{n}\leq\sum_{n\geq0}b_{n}T^{n}$ if and only if
$a_{n}\leq b_{n}.$ The following proposition can be thought of as a modification of
Proposition 3 in \cite{Sul-Vigue} for the non-commutative case.

\begin{proposition}
\label{quotient} Let $(B^*,d_{B})$ be a dga over ${\mathbb Q}$ and let
  $y\in B^k, k\geq 2,$ be
an element such that $d_{B}y=0$ and  $yb\neq 0$ for all $b\in B.$    Then
\begin{equation}
\label{inequality2} S_{H(B/yB)}(T)\leq  (1+T^{k-1})S_{H(B)}(T).
\end{equation}
\end{proposition}

\begin{proof}
 We have an inclusion of cochain complexes
 $s^k B\overset{\iota}{\longrightarrow}B$
   induced  by  the map $B\overset{y\cdot}{\longrightarrow} B,$  $b\rightarrow yb,$
and, consequently,
 the short
exact sequence of cochain complexes
\[
0\longrightarrow s^kB\overset{\iota}{\longrightarrow}B\longrightarrow
B/yB\longrightarrow 0.
\]
Then the proof of the proposition is entirely  analogous to that of \cite[Proposition 7]{saneBetti}.
\end{proof}

\begin{proposition}
\label{rationalone}  Let $H_{_{\mathbb{Q}}}$ be a finitely generated $\mathbb{Q}$-module.
If $\tilde{H}_{_{\mathbb{Q}}}$ requires at least two algebra generators, then the sequence $\left\{\varsigma_{i}(A)\right\} $ grows unbounded.
\end{proposition}

\begin{proof}
Denote  $(B,d_B)=(RH\otimes \bar V_{_{\mathbb{Q}}},d_{\omega}).$
 We will define two sequences   $\mathbf{x}=\left\{x(i) \right\}_{i\geq0}$ and
$\mathbf{y}=\left\{y(j)\right\}_{j\geq 0}$ in $B$ consisting of $d_{B}$-cocycles in the four cases below.
Consider two relations  of the smallest degree in $H _{_{\mathbb{Q}}}$
 \[P_1(x_1,...,x_p)=0\ \  \text{and}  \ \  P_2(y_1,...,y_q)=0.\]
Suppose that

  (i)  All $x_i$ and $y_j$  are even dimensional. Obtain $u_1$  and $u_2$
    from (\ref{basic}) that correspond to the above relations, and
   define
   $\mathbf{x}=\left\{x(i) \right\}_{i\geq0}$ and
$\mathbf{y}=\left\{y(j)\right\}_{j\geq 0}$  by
(\ref{even}) for $w=u_1$ and $w=u_2$ respectively.

(ii) There are  odd dimensional elements $c_1$ and $c_2$ among $x_i$'s and $y_j$'s respectively.
Then define $\mathbf{x}=\left\{x(i) \right\}_{i\geq0}$ and
$\mathbf{y}=\left\{y(j)\right\}_{j\geq 0}$ by

\begin{equation}\label{odd1}
x(i)=1\otimes s^{-1}\left({ c_1}^{\smallsmile_1(i+1)}\right)
\end{equation}
and
 \begin{equation}\label{odd2}
y(j)=1\otimes s^{-1}\left( {c_2}^{\smallsmile_1(j+1)}\right)
\end{equation}
 respectively.

 (iii) There is a single odd dimensional  $x_i$ and  all $y_j$ are even dimensional. Define $\mathbf{x}=\left\{x(i) \right\}_{i\geq0}$
  by    (\ref{odd1}), while
define
$\mathbf{y}=\left\{y(j)\right\}_{j\geq 0}$
as in item (i).
When all $x_i$ are even dimensional and
 a single $y_j$ is odd dimensional,
define $\mathbf{x}=\left\{x(i) \right\}_{i\geq0}$
   as in item (i),  while define
$\mathbf{y}=\left\{y(j)\right\}_{j\geq 0}$ by (\ref{odd2}).

(iv)  There is a single  odd dimensional $x_i$ and a  single  odd dimensional $y_j$ equal to the same element $a\in H_{_{\mathbb Q}}.$
 Then obviously $P_1(x_1,...,x_p)$ admits a representation
 $P_1(x_1,...,x_p)=ab$ for a certain  even dimensional element $b\in H_{_{\mathbb{Q}}}.$ Consequently, the corresponding relation in
 $(RH_{_{\mathbb{Q}}},d)$ given by (\ref{basic})
  has   the form $du=a_0b_0$ for $a_0\in {\mathcal V}^{0,*}$ and $b_0\in R^0H^*.$  Denote $a_1=u$
 and $b_1=-a_1-a_0\smile_1 b_0$ to obtain
$db_1=-b_0a_0.$
 Furthermore, denoting   $b_2=-b_0\smile_1 a_1,$    there are the  induced relations in
$(RH_{_{\mathbb{Q}}},d):$
\[
\begin{array}
[c]{lllll}
da_2=a_0b_1-a_1a_0,   & &  da_3=a_0b_2-a_1a_1-a_2b_0,     \\
db_2=-b_0a_1-b_1b_0,   & &   db_3=  -b_0a_2-b_1b_1+b_2a_0, && a_3,b_3\in V_{_{\mathbb{Q}}}.
\end{array}
\]
 Thus we have
$h(a_3-b_3)=(h^{2}+h^{3})(a_3-b_3)$ with  $h^{2}(a_3-b_3)=-b_0\smile_1 h^2 a_2$ in
$(RH_{_{\mathbb{Q}}},d_h).$
Let $b_0=P'(z_1,...,z_r)$ and  $ h^3(a_3-b_3)=P''(z_{r+1},...,z_m)$  for some $z_j\in V^{0,*}_{_{\mathbb{Q}}},\, 1\leq j\leq m,\,
1\leq r<m.$
Define a complex $(D,d_{D})$ as
   $(D,d_{D})=(B/1\otimes \bar C,d_{D}),$
where $\bar C\subset \bar V_{_{\mathbb{Q}}}$ is a subcomplex (additively) generated by the
expressions     \[\{\bar z_j\,,\overline{z_j\smile_{1}v} \mid v\in V_{_{\mathbb{Q}}},\, 1\leq j\leq m\}.\]
Define $\bar x$ and $\bar y$  in $
\bar {V}_{_{\mathbb{Q}}}/\bar C$ as the projections of the elements $\bar a_0$ and  $\overline{a_3-b_3}$   under
the quotient map $\bar {V}_{_{\mathbb{Q}}}\rightarrow \bar{V}_{_{\mathbb{Q}}}/\bar C$ respectively.
Then $1\otimes \bar x$ and $1\otimes \bar y$ are cocycles in
$(D,d_D).$
Apply formulas (\ref{odd1})--(\ref{odd2}) for $(c_1,c_2)=(x,y)$ to obtain
the sequences $\mathbf{x}=\left\{x(i) \right\}_{i\geq0}$   and
$\mathbf{y}=\left\{y(j)\right\}_{j\geq 0}$ in $D.$ Then
the product classes $\left\{[x(i)]\cdot [y(j)]\right\}
_{i,j\geq0}$ are linearly independent in $H(D\,,d_{D}).$
  Finally,  apply Proposition \ref{quotient}
successively for $y\in \{z_1,...,z_m\}$ to obtain $S_{H(D)} (T)\leq S_{H(B)}(T),$ and then an
application of Proposition \ref{barV} completes the proof.
\end{proof}

\subsection{Proof of Theorem \ref{varsigma}}

In view of Proposition \ref{barV}, the proof reduces to the examination of the
$\Bbbk$-module $H(RH \otimes \bar{V}_{\Bbbk},d_{\omega}).$ If $\tilde{H}_{\Bbbk}$
has a single algebra generator $a,$ then the set $\left\{ \varsigma_{i}(A)\right\}  $ is
bounded since $\varsigma_{i}(A)=2$ (cf. \cite{Halp-Vigue}). If $\tilde{H}_{\Bbbk}$ requires at
least two algebra generators, then the proof follows from    Proposition \ref{barV} and Proposition \ref{one} for $\mu\geq 2,$    and
from Proposition \ref{rationalone}  for $\mu=0.$

\vspace{5mm}

\end{document}